\documentclass[12pt, reqno]{amsart}%
\usepackage[centertags]{amsmath}
\usepackage{amsmath}
\usepackage{amsthm}
\usepackage{graphicx}%
\usepackage{amsfonts}%
\usepackage{amssymb}
\usepackage{color}
\usepackage[all]{xy}
\usepackage{enumerate}
\usepackage[pagewise]{lineno}
\usepackage[colorlinks=true,hyperindex=true,bookmarks=true]{hyperref}
\hypersetup{
    colorlinks,
    citecolor=blue,
    filecolor=purple,
    linkcolor=red,
    urlcolor=blue
} 
\newtheorem{theorem}{Theorem}[section]
\newtheorem{corollary}[theorem]{Corollary}
\newtheorem{lemma}[theorem]{Lemma}

\theoremstyle{definition}

\theoremstyle{remark}
\newtheorem{remark}[theorem]{Remark}

\numberwithin{equation}{section}
\newcommand{\ds}{\displaystyle}
\newcommand{\R}{\mathbb R}
\newcommand{\Com}{\mathbb C}
\newcommand{\N}{\mathbb N}
\newcommand{\Z}{\mathbb Z}

\newcommand{\abs}[1]{\left\vert#1\right\vert}
\newcommand{\set}[2]{\left\{ #1 \,;\, #2 \right\}}
\newcommand{\eps}{\varepsilon}

\begin{document}

\title{On  Bochner's almost-periodicity criterion 
}

\author{Philippe CIEUTAT}
%
%
\setcounter{footnote}{-1}
\renewcommand{\thefootnote}{\alph{footnote}}
\footnote{  Universit\'e Paris-Saclay, UVSQ, CNRS, Laboratoire de math\'ematiques de Versailles, 78000, Versailles, France.  E-mail address: philippe.cieutat@uvsq.fr}

\begin{abstract} 
We  give an extension of  Bochner's criterion for the almost periodic functions.
 By using our main result, we  extend two results of A. Haraux. The first is a generalization  of Bochner's criterion which  is useful for periodic dynamical systems. 
The second is
a characterization of periodic functions in term of Bochner's criterion. 
 \end{abstract}

%
\maketitle

\noindent
{\bf 2020 Mathematic Subject Classification:}  35B10, 35B40, 42A75, 47H20. 

\vskip10mm
%


\section{Introduction} 
\label{i}

The almost periodic functions in the sense of Bohr have been characterized by Bochner by means of a compactness criterion  in the space of the bounded and continuous functions \cite{Bo1, Bo2}.
The Bochner's criterion plays an essential role in the theory and in applications. 
We give a new almost-periodicity criterion for functions with values in a given complete metric space
which is  useful to study  the almost periodicity of solutions of dynamical systems governed by a family of operators with a positive parameter. 
This criterion  is an extension of Bochner's criterion. 
Then Haraux gave a generalization  of Bochner's criterion \cite[Theorem 1]{Ha2}, called {\it a simple almost-periodicity criterion}  which is useful for periodic dynamical systems. 
From our  result, we deduce an extension of this criterion.
We also obtain  an extension of an other result of Haraux 
 which characterizes the periodic functions in terms of the Bochner's criterion \cite{Ha4}.
 In the same spirit, we treat the asymptotically almost periodic case.
\vskip 2 mm
We give a description of this article, the precise definitions will be given in Section \ref{iib}.
Throughout this section $(X,d)$ is a complete metric space.
 An almost periodic function $u:\R\to X$ in the sense of Bohr is characterized by the Bochner's criterion  which is the following: {\it $u$ is bounded and continuous, and  from any real sequence of real numbers $(\tau_n)_n$, there exists a subsequence $(\tau_{\phi(n)})_n$ such that the sequence of functions $(u(t+\tau_{\phi(n)}))_n$ is uniformly convergent on $\R$.}
In Section \ref{ii}, we give two extensions of Bochner's criterion. 
First $u:\R\to X$ is an almost periodic if and if only if in the Bochner's criterion, we impose that the terms of the sequence of real numbers $(\tau_n)_n$ are all {\it positive}.
Second  $u:\R\to X$ is an almost periodic if and if only if in the Bochner's criterion, the convergence of the subsequence of functions $(u(t+\tau_{\phi(n)}))_n$ is uniform only on  {\it $[0,+\infty)$}. 
These improvements are useful to study the almost periodicity of solutions of an evolution equation governed by a family of operators with a positive parameter, in particular for a
$C_0$-semigroup of linear operators or more generally, for an autonomous dynamical system (nonlinear semigroup). 
From our extension of Bochner's criterion,  we give new proofs which are direct and simpler on known results on the almost periodicity of solutions of autonomous dynamic systems.
\vskip 2 mm
Haraux gave a generalization  of Bochner's criterion the called {\it a simple almost-periodicity criterion} \cite[Theorem 1]{Ha2}. This criterion makes it possible to choose in the   Bochner's criterion, the sequence of real numbers $(\tau_n)_n$ in a set of the type $\omega\Z$ which is very useful for periodic dynamical systems. 
From our extension  of Bochner's criterion, in Section \ref{iii}, we deduce an improvement of this result.
An asymptotically almost periodic function $u:\R^+\to X$ is a perturbation of almost periodic.
A such function  is characterized by a property of the type of the Bochner's criterion. 
In the same spirit, we extend this characterization of  asymptotically almost periodic functions. 
Then we apply these results to study the almost periodicity of solutions of periodic dynamical systems. 
\vskip 2 mm
Bochner's criterion can also be expressed in terms of the relative compactness of the set $\{u(\cdot+\tau) ; \tau\in\R \}$   in a suitable set of continuous functions.
A periodic function is a special  case of almost periodic function.
A direct consequence of \cite[Proposition 2]{Ha4} given by Haraux  characterizes a periodic function in terms of the Bochner's criterion. This characterization   is the following:
$u:\R\to X$ is continuous is periodic if and if only if the set $\{u(\cdot+\tau) ; \tau\in\R \}$   is compact. 
In Section \ref{iv}, By using our improvement of Bochner's criterion, we give an extension  of  the Haraux's characterization of periodic functions. 
We will also give a result on asymptotically periodic functions of the type of Haraux result described above. Then we apply these results to study the periodicity of solutions of autonomous dynamical systems.

\section{Notation}
\label{iib}

Let us now give some notations,  definitions and properties which will be used. 
\vskip 1 mm
Throughout this section $(X,d)$ is a complete metric space.
$\R$, $\Z$ and $\N$ stand respectively for the real numbers, the integers and  the natural integers. 
We denote by $\R^+ := \{t\in\R  ; t\geq  0\}$. Let $E$ be  a topological space. We denote  by $C(E,X)$ the space of all continuous functions from $E$ into $X$.
When $J=\R$ or $J=\R^+$, we denote  by $BC(J,X)$  the space of all bounded and continuous  functions from $J$  into $X$ equipped with the sup-distance, denoted by  $\ds d_{\infty}(u, v) :=\sup_{t\in \R}d(u(t),v(t))$ when  $J=\R$ and $\ds d_{\infty,+}(u, v) :=\sup_{t\geq0}d(u(t),v(t))$) when $J=\R^+$ for $u$, $v \in BC(J,X)$. The metric spaces $(BC(\R,X),d_{\infty})$ and  $(BC(\R^+,X),d_{\infty,+})$) are complete.
\vskip 5 mm
We now give some definitions and properties on almost periodic, asymptotically almost periodic functions with values in a given complete metric space. 
\vskip 2 mm
A subset $D$ of $\R$ (respectively of $\R^+$) is said to be {\it relatively dense}  if
there exists $\ell>0$ such that $D\cap [\alpha,\alpha+\ell]\not=\emptyset$ for all $\alpha\in\R$ (respectively $\alpha\geq0$).
 A continuous function $u:\R\to X$ is said to be {\it almost periodic (in the sense of Bohr)} if  for each $\varepsilon > 0$, the set of $\varepsilon$-almost periods:
$\displaystyle\mathcal{P}(u,\varepsilon)=\set{\tau\in\R}{\sup_{t \in \mathbb{R}} d(u (t + \tau) , u (t))\leq  \varepsilon}$
 is relatively dense in $\R$.  An almost periodic function $u$ has its range $u(\R)$  relatively compact, that is its closure denoted by $\rm{cl}\left(u(\R)\right)$ is a compact set of $(X,d)$.
 We denote the space of all such functions by $AP(\R,X)$. It is a closed metric subspace of $(BC(\R,X),d_{\infty})$.
An almost periodic function $u$ is \textit{uniformly recurrent}, that is there exists a sequence of real numbers $(\tau_n)_n$ such that $\ds\lim_{n\to+\infty}\sup_{t\in\R}d(u(t+\tau_n),u(t))=0$ and $\ds\lim_{n\to+\infty}\tau_n=+\infty$. To see that consider the Bohr's definition of $u\in AP(\R,X)$,  then the  set of $\frac{1}{n}$-almost periods  satisfies $\displaystyle\mathcal{P}(u,\frac{1}{n})\cap[n,+\infty)\not=\emptyset$, for each integer $n>0$.
A useful characterization of almost periodic functions was given by Bochner. 
The Bochner's criterion which may be found in \cite[Bochner's theorem, p. 4]{LZ}  in the context of metric spaces. Before to cite this  criterion, we need to introduce the translation mapping of a function of $BC(\R,X)$. For $\tau\in\R$ and $u\in BC(\R,X)$, we define {\it the translation mapping} $T_{\tau}u\in BC(\R,X)$ by $T_{\tau}u(t) = u(t+\tau)$ for $t\in\R$. 
\begin{theorem}[Bochner's criterion] 
For $u\in BC(\R,X)$, the following statements are equivalent.
\vskip 2 mm
{\bf i)} $u\in AP(\R,X)$.
\vskip 2 mm
{\bf ii)} The set $\{T_{\tau}u ; \tau\in\R \}$ is relatively compact in $(BC(\R,X),d_{\infty})$. 
\end{theorem}
Haraux gave a generalization   of Bochner' criterion the called {\it a simple almost-periodicity criterion} \cite[Theorem 1]{Ha2} which is useful for periodic dynamical systems. 
\begin{theorem}[Haraux's criterion] \label{th3}
Let $D$ be a relatively dense subset of $\R$.
The following statements are equivalent for $u\in BC(\R,X)$.
\vskip 2 mm
{\bf i)} $u\in AP(\R,X)$.
\vskip 2 mm
{\bf ii)} The set $\{T_{\tau}u ; \tau\in D \}$ is relatively compact in $(BC(\R,X),d_{\infty})$.
\end{theorem}
Periodic functions, which are a special case of almost periodic functions,  are also characterized in terms of Bochner's criterion. This criterion is a direct consequence of a result of Haraux.
\begin{theorem}\cite[Consequence of Proposition 2]{Ha4}  \label{th1}
The following statements are equivalent for $u\in BC(\R,X)$.
\vskip 2 mm
{\bf i)} $u$ is periodic.
\vskip 2 mm
{\bf ii)} The set $\{T_{\tau}u ; \tau\in\R \}$ is a compact set of $(BC(\R,X),d_{\infty})$. 
\end{theorem}
\vskip 2 mm
For some preliminary results on almost periodic functions with values in a given complete metric space, we refer to the book of Levitan-Zhikov \cite{LZ} and in the special  case of Banach spaces to the book of Amerio-Prouse \cite{Am-Pr}.
\vskip 3 mm
The notion of asymptotic almost periodicity was first introduced by Fr\'echet \cite{Fr} in 1941 in the case where $X=\Com$.
A continuous function $u:\R^+\to X$ is said to be {\it asymptotically almost  periodic} if there exists  $v\in AP(\R,X)$  such that $\ds\lim_{t\to\infty}d(u(t),v(t))=0$.
An asymptotic almost periodic function $u$ has its range $u(\R^+)$  relatively compact.
 We denote the space of all such functions by $AAP(\R^+,X)$. It is a closed metric subspace of $(BC(\R^+,X),d_{\infty,+})$.
An asymptotic almost periodicity function $u:\R^+\to X$ is characterized by {\it $u\in APP(\R^+,X)$ if and only if $u\in C(\R^+,X)$ and  for each $\varepsilon > 0$, there exists $M\geq0$ such that the  
$\displaystyle\set{\tau\geq0}{\sup_{t \geq M} d(u (t + \tau) , u (t))\leq  \varepsilon}$
 is relatively dense in $\R^+$} \cite[Theorems 1.3]{RS2}.
In the context of metric spaces, Ruess and Summers give  a characterization of asymptotically almost periodic functions in the spirit of Bochner's criterion.
To prove this characterization, Ruess and Summers use results from the paper \cite{RS3} by the same authors.
For $\tau\geq0$ and $u\in BC(\R^+,X)$, we define {\it the translation mapping} $T^+_{\tau}u\in BC(\R^+,X)$ by $T^+_{\tau}u(t) = u(t+\tau)$ for $t\geq0$. 
\begin{theorem}\cite[a part of Theorems 1.2 \& 1.3]{RS2} \label{cor2}
Let $(X,d)$ be a complete metric space. 
For $u\in BC(\R^+,X)$, the following statements are equivalent.
\vskip 2 mm
{\bf i)} $u\in AAP(\R^+,X)$.
\vskip 2 mm
{\bf ii)} The set $\{T^+_{\tau}u ; \tau\geq0 \}$ is relatively compact in $(BC(\R^+,X),d_{\infty,+})$.
\end{theorem}
For some preliminary results on asymptotically almost periodic functions, we refer to the book of Yoshizawa  \cite{Yo} in  the case where $X$ is a finite dimensional space, to the book of Zaidman \cite{Za} where $X$ is a Babach space and to Ruess and Summers \cite{RS1, RS2, RS3} in the general case: $X$ is a complete metric space.

\section{An improvement of Bochner's criterion}
\label{ii}

An almost periodic function is  characterized by the Bochner's criterion, recalled in Section \ref{iib}.
Our main result is an extension of Bochner's criterion. 
Then we deduce  new  proofs which are direct and simpler on known results on the solutions of autonomous dynamic systems.
Before to state our extension of Bochner's criterion, we need to introduce the restriction operator $R:BC(\R,X)\to BC(\R^+,X)$ defined by $R(u)(t):=u(t)$ for $t\geq0$ and $u\in  BC(\R,X)$.

\begin{theorem} \label{th2}
Let $(X,d)$ be a complete metric space. 
For $u\in BC(\R,X)$ the following statements are equivalent.
\vskip 2 mm
{\bf i)} $u\in AP(\R,X)$.
\vskip 2 mm
{\bf ii)} The set $\{T_{\tau}u ; \tau\geq0 \}$ is relatively compact in $(BC(\R,X),d_{\infty})$.
\vskip 2 mm
{\bf iii)} The set $\{R(T_{\tau}u) ; \tau\in\R \}$ is relatively compact in $(BC(\R^+,X),d_{\infty,+})$. 
\end{theorem}

In our results, the compactness and the relative compactness of a set often intervene. To prove them, we will often use the following result whose proof is obvious. Recall that a set $A$ of a metric space $(E,d)$ is relatively compact if its closure denoted by $\rm{cl}\left(A\right)$ is a compact set of $(E,d)$.

\begin{lemma} \label{lem2}
Let $E$ be a set,  $(G_1,d_1)$ and $(G_2,d_2)$ be two metric spaces. 
Let $u:E\to G_1$ and $v:E\to G_2$ be two functions. Assume   there exists $M>0$ such that
\begin{equation*}
\forall x_1 , x_2 \in E , \quad d_1(u(x_1),u(x_2)) \leq M d_2(v(x_1),v(x_2)) . 
\end{equation*}
Then the following statements hold.
\vskip 2 mm
{\bf i)} If the metric space $(G_1,d_1)$ is complete and $v(E)$ is relatively compact in $(G_2,d_2)$, then $u(E)$ is relatively compact in $(G_1,d_1)$.
\vskip 2 mm
{\bf ii)} If $v(E)$ is a compact set of $(G_2,d_2)$, then $u(E)$ is a compact set of $(G_1,d_1)$.
\end{lemma}

\begin{proof}[Proof of Theorem \ref{th2}] 
{\bf i)} $\Longrightarrow$ {\bf iii)}. It is obvious by using the Bochner's criterion and the continuity of the restriction operator $R$.
\vskip 2 mm
{\bf iii)} $\Longrightarrow$ {\bf ii)}. 
The set $u(\R) = \{R(T_{\tau}u)(0) ;  \tau\in\R\}$ is relatively compact in $X$ as the range of $\{R(T_{\tau}u) ; \tau\in\R \}$ by the continuous evaluation map at $0$ from $BC(\R^+, X)$ into $X$.
By assumption, $\mathcal{H}:=\rm{cl}\left(\set{R(T_{t}u)}{t\in\R}\right)$ is a compact set of $(BC(\R^+,X),d_{\infty,+})$. For all $\tau\geq0$, we define $\phi_\tau:\mathcal{H}\to X$ by $\phi_\tau(h)=h(\tau)$. The functions $\phi_\tau$ are $1$-Lipschitz continuous and for each $t\in\R$, the set $\set{\phi_\tau(R(T_{t}u))=u(\tau+t)}{\tau\geq0}$ is included in the relatively compact set  $u(\R)$.  By density of $\set{R(T_{t}u)}{t\in\R}$ in $\mathcal{H}$ and the continuity of $\phi_\tau$, it follows that $\set{\phi_\tau(h)}{\tau\geq0}$ is relatively  compact in $X$ for each $h\in \mathcal{H}$. 
According to Arzel\` a-Ascoli's theorem \cite[Theorem 3.1, p. 57]{La},
 the set $\set{\phi_\tau}{\tau\geq0}$ is relatively compact in $C(\mathcal{H},X)$ equipped with the sup-norm denoted by $d_C$. 
  From the density of $\set{R(T_{t}u)}{t\in\R}$ in $\mathcal{H}$ and   the continuity of $\phi_\tau$,  we deduce that for $\tau_1$ and $\tau_2\geq0$,
 $\ds \sup_{h\in\mathcal{H}}d(\phi_{\tau_1}(h),\phi_{\tau_2}(h)) = \sup_{t\in\R}d\left(\phi_{\tau_1}(R(T_{t}u)),\phi_{\tau_2}(R(T_{t}u))\right) = \sup_{t\in\R}d\left(u(\tau_1+t),u(\tau_2+t)\right)  = \sup_{t\in\R}d\left(
 T_{\tau_1}u(t),T_{\tau_2}u(t)\right)$, then 
$\ds d_{C}(\phi_{\tau_1},\phi_{\tau_2}) = d_{\infty}\left(T_{\tau_1}u,T_{\tau_2}u\right)$.
From Lemma \ref{lem2},
 it follows that  $\{T_{\tau}u ; \tau\geq0 \}$ is relatively compact in the complete metric space $(BC(\R,X),d_{\infty})$ since $\set{\phi_\tau}{\tau\geq0}$ is also one in $(C(\mathcal{H},X),d_{C})$.
\vskip 2 mm
{\bf ii)} $\Longrightarrow$ {\bf i)}. For $\tau_1$, $\tau_2\geq0$, $\ds d_{\infty}(T_{\tau_1}u,T_{\tau_2}u):= \sup_{t\in\R}d(u(\tau_1+t),u(\tau_2+t))$.
Replacing $t$ by $t-\tau_1-\tau_2$ in the upper bound, we get $\ds d_{\infty}(T_{\tau_1}u,T_{\tau_2}u) = d_{\infty}(T_{-\tau_1}u,T_{-\tau_2}u)$. Then the set $\{T_{\tau}u ; \tau\leq0 \}= \{T_{-\tau}u ; \tau\geq0 \}$ is relatively compact in $BC(\R,X)$ since $\{T_{\tau}u ; \tau\geq0 \}$ is also one. Therefore the set $\{T_{\tau}u ; \tau\in\R \}$ is relatively compact in $BC(\R,X)$ as the union of two relatively compact sets in $BC(\R,X)$. According to Bochner's criterion, $u\in AP(\R,X)$.
\end{proof}

The connection between the almost periodicity of a solution of a dynamical system and its stability is well known (see the monograph by Nemytskii \& Stepanov \cite[Ch. 5]{NS}.
This weakened form of Bochner's criterion: Theorem \ref{th2} makes it possible to obtain direct and simpler proofs on these questions. Let us start by recalling some definitions on dynamical systems.
\vskip 1 mm
A \textit{dynamical system} or \textit{nonlinear semigroup} on a complete metric space $(X,d)$  is a one parameter family $\left( S(t) \right)_{t\geq0}$ of maps from $X$ into itself such that {\bf i)} $S(t)\in C(X,X)$ for all $t \geq 0$, {\bf ii)} $S(0)x=x$ for all $x \in X$, {\bf iii)}  $S(t+s) = S(t) \circ S(s)$ for all $s, t \geq 0$ and {\bf iv)} the mapping $S(\cdot)x\in C([0,+\infty),X)$ for all $x \in X$.
\vskip 1 mm
For each $x\in X$, \textit{the positive trajectory of $x$} is the map $S(\cdot)x:\R^+\to X$.
A function $u:\R\to X$ is called \textit{a complete trajectory} if we have $u(t+\tau)=S(\tau)u(t)$, for all $t\in\R$ and $\tau\geq0$. 
\vskip 2 mm
We will need a notion of Lagrange-type stability to ensure that a  solution with a relatively compact range is almost periodic. Recall that 
 $\left( S(t) \right)_{t\geq0}$ is  equicontinuous on a compact set $K$ of $X$, if forall $\eps>0$, there exists $\delta>0$, such that 
 $$\forall x_1, x_2\in K, \ds d(x_1,x_2)\leq\delta\implies\sup_{t\geq0}d(x_1,x_2)\leq\eps .$$
\vskip 2 mm
Using Theorem \ref{th2}, we give a new proof which is direct and simpler of the following result which can be found in \cite[Theorem 4.3.2, p. 51]{Ha3} or partly in \cite[Markov's theorem, p. 10]{LZ}.

\begin{corollary} \label{cor3}
Let $\left( S(t) \right)_{t\geq0}$ be a dynamical system  on a complete metric space $(X,d)$ and $u$ be a complete trajectory such that $u(\R)$ is relatively compact.
Then $u$ is almost periodic if and only if $\left( S(t) \right)_{t\geq0}$ is  equicontinuous on $\text{cl}\left(u(\R)\right)$ the closure of $u(\R)$.
\end{corollary}

\begin{proof}
Let us denote the compact set $K:=\text{cl}\left(u(\R)\right)$.
It follows by density of $u(\R)$ in $K$ and the continuity of $S(t)$, that $\{S(t) x ; t\geq0 \}\subset K$ for each $x\in K$. According to Arzel\` a-Ascoli's theorem, $\left( S(t) \right)_{t\geq0}$ is  equicontinuous on  $K$ if and only if $\left( S(t) \right)_{t\geq0}$ is relatively compact in  $C(K,X)$.
From Theorem \ref{th2}, we have
$u\in AP(\R,X)$  if and only if $\{T_{\tau}u ; \tau\geq0 \}$ is relatively compact in $BC(\R,X)$.
Then it remains to prove that $\left( S(t) \right)_{t\geq0}$ is relatively compact in  $C(K,X)$  equipped with the sup-norm 
 if and only if $\{T_{\tau}u ; \tau\geq0 \}$ is relatively compact in $(BC(\R,X),d_{\infty})$.
This results from the following equalities, for $\tau_1$ and $\tau_2\geq0$,
$\ds\sup_{t\in\R}d\left(T_{\tau_1}u(t),T_{\tau_2}u(t)\right) = \sup_{t\in\R} d\left(S(\tau_1)u(t),S(\tau_2)u(t)\right) 
 = \sup_{x\in K} d\left(S(\tau_1)x,S(\tau_2)x\right)$ and Lemma \ref{lem2}.
\end{proof}

\begin{remark}
{\bf a)} The condition of equicontinuity 
 required by Corollary \ref{cor3} is  satisfied by a bounded dynamical system
: $d\left( S(t)x_1,S(t)x_2\right) \leq Md\left(x_1,x_2\right)$ for some $M\geq1$ and in particular for a $C_0$ semigroup of contractions.
In this case, the almost periodicity of  a complete trajectory $u$ having a relatively compact range
results from Corollary \ref{cor3}. We can also obtain this result with
the implication iii) $\Longrightarrow$ i) of Theorem \ref{th2} and
the  inequality
$\ds\sup_{t\geq0}d(R(T_{\tau_1}u)(t),R(T_{\tau_2}u)(t))=\sup_{t\geq0}d\left( S(t)u(\tau_{1}),S(t)u(\tau_{2}\right)\leq M d(u(\tau_1),u(\tau_2))$ for $\tau_1$, $\tau_2\in\R$.
\vskip 2 mm
{\bf b)} For  a bounded $C_0$-semigroup $\left( S(t) \right)_{t\geq0}$, the main result of Zaidman  \cite{Za1} asserts that a positive trajectory $u$ with relatively compact range satisfies a condition called
\textit{the generalized normality property in Bochner's sense}, without concluding that $u$ is almost periodic. This condition is nothing but hypothesis iii) of Theorem \ref{th2}, so $u$ is almost periodic.
\end{remark}

Using Theorem \ref{cor2}, we give a new proof which is direct and simpler of the following result which can be found in \cite[Theorem 2.2, p. 149]{RS2}.

\begin{corollary} \label{cor6}
Let $\left( S(t) \right)_{t\geq0}$ be a dynamical system  on a complete metric space $(X,d)$ and $u$ be a positive trajectory such that $u(\R^+)$ is relatively compact.
Then $u$ is asymptotically almost periodic if and only if $\left( S(t) \right)_{t\geq0}$ is  equicontinuous on $\text{cl}\left(u(\R^+)\right)$.
\end{corollary}

\begin{proof}
The proof is analogous to that of Corollary \ref{cor3}, using Corollary  \ref{cor2}  instead of Theorem \ref{th2} and by replacing $\R$ by $\R^+$ and $AP(\R,X)$ by $AAP(\R^+,X)$.
\end{proof}

\section{An improvement of Haraux's criterion}
\label{iii}

Haraux gave a generalization  of Bochner'scriterion \cite[Theorem 1]{Ha2}, the called {\it a simple almost-periodi\-city criterion} which is useful for periodic dynamical systems. From our main result, Theorem \ref{th2}, we deduce an extension of the Haraux's criterion, recalled in Section \ref{iib}.
In the same spirit, we extend the well-known characterization of asymptotically almost periodic functions.
To end this section, we give an exemple of application on
a periodic dynamical system. 
\vskip 2 mm
We give an extension of the Haraux's criterion (see Theorem \ref{th3}).
Recall that we denote by $R$ the restriction operator $R:BC(\R,X)\to BC(\R^+,X)$ defined by $R(u)(t):=u(t)$ for $t\geq0$ and $u\in  BC(\R,X)$.

\begin{corollary} \label{prop2}
Let $(X,d)$ be a complete metric space. For $u\in BC(\R,X)$ the following statements are equivalent.
\vskip 2 mm
{\bf i)} $u\in AP(\R,X)$.
\vskip 2 mm
{\bf ii)} The set $\{T_{\tau}u ; \tau\in D \}$ is relatively compact in $(BC(\R,X),d_{\infty})$ where $D$ be a relatively dense subset of $\R^+$.
\vskip 2 mm
{\bf iii)} The set $\{R(T_{\tau}u) ; \tau\in D \}$ is relatively compact in $(BC(\R^+,X),d_{\infty,+})$ where $D$ be a relatively dense subset of $\R$.
\end{corollary}

\begin{remark} 
Our main result, Theorem \ref{th2}  is obviously a particular case of Corollary \ref{prop2}. But to  present our results, it was easier to start with Theorem \ref{th2}. To prove Corollary \ref{prop2}, we use Haraux's criterion and Theorem \ref{th2}.
\end{remark}

\begin{proof}[Proof of Corollary \ref{prop2}]
{\bf i)} $\Longrightarrow$ {\bf iii)}. It is a consequence of Theorem \ref{th2}. 
\vskip 2 mm
{\bf iii)} $\Longrightarrow$ {\bf ii)}. 
To establish this implication, using Theorem \ref{th2}, it suffices to show that assertion iii) implies that $\{R(T_{\tau}u) ; \tau\in\R \}$ is relatively compact in $BC(\R^+,X)$.  The proof of this last implication is a slight adaptation of those of those of the Haraux's criterion given in \cite[Theorem 1]{Ha2}. A similar proof will be detailed in the following result as there will be technical issues. To demonstrate that $\{R(T_{\tau}u) ; \tau\in\R \}$ is relatively compact, it suffices in the proof of ii) $\Longrightarrow$ i) of Corollary \ref{cor1}
to take $\ell=0$ and  replace $\{T^+_{\tau}u ; \tau\in D \}$ by $\{R(T_{\tau}u) ; \tau\in D \}$.
\vskip 2 mm
{\bf ii)} $\Longrightarrow$ {\bf i)}. 
For $\tau_1$, $\tau_2\geq0$, $\ds d_{\infty}(T_{\tau_1}u,T_{\tau_2}u)= \sup_{t\in\R}d(u(\tau_1+t),u(\tau_2+t))$.
Replacing $t$ by $t-\tau_1-\tau_2$ in the upper bound, we get $\ds d_{\infty}(T_{\tau_1}u,T_{\tau_2}u) = d_{\infty}(T_{-\tau_1}u,T_{-\tau_2}u)$.
Then the set $\{T_{\tau}u ; \tau\in -D \}= \{T_{-\tau}u ; \tau\in D \}$ is relatively compact in $BC(\R,X)$ since $\{T_{\tau}u ; \tau\in D \}$ is also one. Therefore the set $\{T_{\tau}u ; \tau\in D\cup(-D) \}$ is relatively compact in $BC(\R,X)$. Moreover $D\cup(-D)$ is a relatively dense subset of $\R$.
According to Haraux's criterion, we have $u\in AP(\R,X)$.
\end{proof}

We extend Theorem \ref{cor2}, the  well-known characterization of asymptotically almost periodic functions. 
For $\tau\in\R^+$ and $u\in BC(\R^+,X)$, we define {\it the translation mapping} $T^+_{\tau}u\in BC(\R^+,X)$ by $T^+_{\tau}u(t) = u(t+\tau)$ for $t\geq0$.

\begin{corollary} \label{cor1}
Let $(X,d)$ be a complete metric space and  let 
$D$ be a relatively dense subset of $\R^+$.
 For $u\in BC(\R^+,X)$ the following statements are equivalent.
\vskip 2 mm
{\bf i)} $u\in AAP(\R^+,X)$.
\vskip 2 mm
{\bf ii)} The set $\{T^+_{\tau}u ; \tau\in D \}$ is relatively compact in $(BC(\R^+,X),d_{\infty,+})$.
\end{corollary}

\begin{remark}
To establish implication ii) $\Longrightarrow$  i), by using Theorem \ref{cor2}, it suffices to prove that assertion  ii) implies that $\{T^+_{\tau}u ; \tau\geq0 \}$ is relatively compact in $(BC(\R^+,X),d_{\infty,+})$. The proof of this last implication is an adaptation of those of the Haraux's criterion. But contrary to the proof of  implication  iii) $\Longrightarrow$  ii) in  Corollary \ref{prop2}, there are technical issues. These technical difficulties come from the fact that when $D$ is a relatively dense subset in $\R^+$, the sets $D$ and $[t-\ell,t]$ can be disjoint
for some $0\leq t\leq\ell$. For this reason we give the complete proof of this implication.
\end{remark}

\begin{proof}[Proof of Corollary \ref{cor1}]
{\bf i)} $\Longrightarrow$ {\bf ii)}. It is a consequence 
of Theorem \ref{cor2}. 
\vskip 2 mm
{\bf ii)} $\Longrightarrow$ {\bf i)}. We will prove that assumption  ii) implies $\{T^+_{\tau}u ; \tau\geq0 \}$ is relatively compact in $BC(\R^+,X)$, then we conclude by using Theorem \ref{cor2}. 
The subset $D$  being  relatively dense in $\R^+$,  there exists $\ell>0$ such that $D\cap [\alpha,\alpha+\ell]\not=\emptyset$ for all $\alpha\geq0$.
\vskip2mm
\textbullet\
{\it We prove that $u$ is uniformly continuous on $[\ell,+\infty)$.} 
Let us fix $\eps>0$. By assumption the set 
$\{T^+_{\tau}u ; \tau\in D \}$
 is  in particular relatively compact in
$C([0,2\ell],X)$, then it is uniformly equicontinuous on $[0,2\ell]$, that is there exists $\delta>0$ such that 
\begin{equation}
\label{eq2}
s_1, s_2\in [0,2l], \quad \abs{s_1-s_2}\leq\delta\implies\sup_{\tau\in D}d(u(s_1+\tau),u(s_2+\tau))\leq\eps .
\end{equation}
Let $t_1$, $t_2$ be two real numbers such that $t_1$, $t_2\geq\ell$ and $\abs{t_1-t_2}\leq\delta$. We can assume without loss of generality that $\ell\leq t_1\leq t_2\leq t_1 + \ell$.  We have $D\cap [t_1-\ell,t_1]\not=\emptyset$ since $t_1-\ell\geq0$, then there exists $\tau\in D$ such that 
$0 \leq t_1- \tau\leq t_2 -\tau \leq 2l$. Taking account  \eqref{eq2}, we deduce that 
$\ds d(u(t_1),u(t_2)) = d(u((t_1- \tau) + \tau),u((t_2- \tau) + \tau)) \leq\eps$.
Hence $u$ is uniformly continuous on $[\ell,+\infty)$. 
\vskip2mm
\textbullet\
{\it We prove that $\{T^+_{\tau}u ; \tau\geq\ell \}$ is relatively compact in $BC(\R^+,X)$
.} Let $(t_n)_n$ be a sequence of real numbers such that $t_n\geq\ell$. We have $D\cap [t_n-\ell,t_n]\not=\emptyset$ for each $n\in\N$, since $t_n-\ell\geq0$, then there exist $\tau_n\in D$ and $\sigma_n\in [0,l]$ such that $t_n=\tau_n+\sigma_n$.
By compactness of the sequences $(\sigma_n)_n$ in $[0,\ell]$ and $(T^+_{\tau_n}u)_n$ in  $BC(\R^+,X)$, it follows that
$\ds\lim_{n\to+\infty}\sigma_n=\sigma$  and  $\ds\lim_{n\to+\infty}\sup_{t\geq0}d(u(t+\tau_n),v(t))=0$  (up to a subsequence).
From the following inequality 
$$\sup_{t\geq0}d(u(t_n+t),v(\sigma +t)) \leq \sup_{t\geq0}d(u(\tau_n+\sigma_n+t),u(\tau_n + \sigma+t) + \sup_{t\geq0}d(u(\tau_n + t),v(t))$$
and the uniform continuity of $u$, we deduce that
$\ds \lim_{n\to+\infty}\sup_{t\geq0}d(u(t_n+t),v(\sigma +t))=0$.
Then $\{T^+_{\tau}u ; \tau\geq\ell \}$ is relatively compact in $BC(\R^+,X)$. 
\vskip2mm
\textbullet\
{\it We prove that $u\in AAP(\R^+,X)$.}  
The function $u$ is uniformly continuous on $\R^+$, since $u$ is continuous on $[0,\ell]$ and uniformly continuous on $[\ell,+\infty)$. Then the map $\hat{u}:\R^+\to BC(\R^+,X)$ defined by $\hat{u}(\tau)=T^+_{\tau}u$ for $\tau\geq0$ is continuous, consequently the  set $\{T^+_{\tau}u ; 0\leq\tau\leq\ell \}$ is relatively compact in $BC(\R^+,X)$.
The set $\{T^+_{\tau}u ; \tau\geq0 \}$ is relatively compact in $BC(\R^+,X)$ as the union of two relatively compact sets. According to Theorem \ref{cor2}, $u\in AAP(\R,X)$.
\end{proof}

Using corollaries \ref{prop2} and \ref{cor1}, we give a  proof which is direct and simpler of the following result which can be found in \cite{Ha1, Ha2, Ha3}. Before we recall some definitions on  process.
\vskip 1 mm
A \textit{ process} on a complete metric space $(X,d)$  according to Dafermos \cite{Da} is a two parameter family $U(t,\tau)$ of maps from $X$ into itself defined for
$(t,\tau)\in\R\times\R^+$ and such that {\bf i)} $U(t,0)x=x$ for all $(t,x) \in \R\times X$, 
{\bf ii)}  $U(t,\sigma+\tau) = U(t+\sigma,\tau) \circ U(t,\sigma)$ for all $(t,\sigma,\tau)\in\R\times\R^+\times\R^+$
 and {\bf iii)} the mapping $U(t,\cdot)x\in C([0,+\infty),X)$ for all $(t,x) \in \R\times X$.
\vskip 1 mm
For each $x\in X$, \textit{the positive trajectory starting of $x$} is the map $U(0,\cdot)x:\R^+\to X$.
A function $u:\R\to X$ is called \textit{a complete trajectory} if we have $u(t+\tau)=U(t,\tau)u(t)$ for all $(t,\tau)\in\R\times\R^+$. 
\vskip 1 mm
A process $U$ is said \textit{$\omega$-periodic} ($\omega>0$) if $U(t+\omega,\tau)=U(t,\tau)$ for all $(t,\tau)\in\R\times\R^+$.
\vskip 1 mm
A process $U$ is said \textit{bounded}  if we have for some $M\geq 1$ for all $(\tau,x_1,x_2)\in\R^+\times X\times X$ $d\left( U(0,\tau)x_1,U(0,\tau)x_2\right) \leq Md\left(x_1,x_2\right)$.

\begin{corollary}\cite{Ha1, Ha2}, \cite[Th\'eor\`eme 6.4.6, p. 84]{Ha3}
Let U be a $\omega$-periodic process on a complete metric space $(X,d)$. If $U$ is bounded, then
  the following statements hold.
\vskip 2 mm
{\bf i)} If $u$ is a complete trajectory of $U$ such that $u(-\omega\N)$ is relatively compact, then $u$ is almost periodic.
\vskip 2 mm
{\bf ii)} If $u$ is a positive trajectory of $U$ such that $u(\omega\N)$ is relatively compact, then $u$ is asymptotically almost periodic.
\end{corollary}

\begin{proof}
\vskip 2 mm
{\bf i)}  The process $U$ is $\omega$-periodic, then we have  $u(n\omega) = U(-m\omega,(n+m)\omega)u(-m\omega) = U(0,(n+m)\omega)u(-m\omega)$ for all $n$, $m\in\N$. 
 From the boundedness assumption on $U$, we deduce that    
 $\ds d\left(u(n\omega),u(m\omega)\right) \leq Md\left(u(-m\omega),u(-n\omega)\right)$, then 
 $u(\omega\N)$ is relatively compact since $u(-\omega\N)$ is also one, therefore $u(\omega\Z)$ is relatively compact.
 From assumptions on the process $U$, it follows that for all $n$, $m\in\Z$,
\begin{equation} \label{eq1}
\sup_{\tau\geq0}d\left(u(\tau+n\omega),u(\tau+m\omega)\right) \leq Md\left(u(n\omega),u(m\omega)\right) .
\end{equation}
From Lemma \ref{lem2},
 $\{R(T_{n\omega}u) ; n\in \Z \}$ is relatively compact in $(BC(\R^+,X),d_{\infty,+})$ since $u(\omega\Z)$ is also one in $(X,d)$. We conclude with Corollary \ref{prop2} by setting $D=\omega\Z$.
\vskip 2 mm
{\bf ii)} For all $n$, $m\in\N$, \eqref{eq1} holds on the positive trajectory $u$,
then from Lemma \ref{lem2},
 $\{T^+_{n\omega}u ; n\in\N \}$ is relatively compact in $(BC(\R^+,X),d_{\infty,+})$ since $u(\omega\N)$ is also one in $(X,d)$. We conclude with Corollary \ref{cor1} by setting $D=\omega\N$.
\end{proof}

\section{Bochner's criterion in the periodic case}
\label{iv}

Periodic functions are a special case of almost periodic functions. 
Haraux gave a characterization of  periodic functions in terms of Bochner's criterion which is recalled in Section \ref{iib}.
This criterion  is a direct consequence of \cite[Proposition 2]{Ha4}. Haraux established a general result \cite[Th\'eor\`eme 1]{Ha4} implying as a special case a characterization of periodic functions  and the fact 
that any compact trajectory of a one-parameter continuous group is automatically periodic. 
\vskip 2 mm
In this section, we give an extension of this characterization of periodic functions in the spirit of the main result of this article. We also treat the asymptotically periodic case. Then we apply these results to study the periodicity of solutions of  dynamical systems. 
\vskip 2 mm
Recall that we denote by $R$ the restriction operator $R:BC(\R,X)\to BC(\R^+,X)$ defined by $R(u)(t):=u(t)$ for $t\geq0$ and $u\in  BC(\R,X)$.

\begin{corollary} \label{cor4}
Let $(X,d)$ be a complete metric space. 
For $u\in BC(\R,X)$ the following statements are equivalent.
\vskip 2 mm
{\bf i)} The function $u$ is $\omega$-periodic ($\omega>0$).
\vskip 2 mm
{\bf ii)} The set $\{T_{\tau}u ; \tau\geq0 \}$ is a compact set of $(BC(\R,X),d_{\infty})$.
\vskip 2 mm
{\bf iii)} The set $\{R(T_{\tau}u) ; \tau\in\R \}$ is a compact set of $(BC(\R^+,X),d_{\infty,+})$. 
\end{corollary}

\begin{proof}
{\bf i)} $\Longrightarrow$ {\bf ii)}.
From assumption, it follows that the function $\tau\to T_{\tau}u$ from $\R$ into $BC(\R,X)$  is  continuous and $\omega$-periodic.
 Then the
$\{T_{\tau}u ; \tau\geq0 \}=\{T_{\tau}u ; 0\leq\tau\leq\omega \}$ is a compact set of $(BC(\R,X),d_{\infty})$ as the range of a compact set by a continuous  map.
\vskip 2 mm
{\bf ii)} $\Longrightarrow$ {\bf i)}. 
For $\tau_1$, $\tau_2\in\R$, $\ds d_{\infty}(T_{\tau_1}u,T_{\tau_2}u):= \sup_{t\in\R}d(u(\tau_1+t),u(\tau_2+t))$, we get $\ds d_{\infty}(T_{\tau_1}u,T_{\tau_2}u) = d_{\infty}(T_{-\tau_1}u,T_{-\tau_2}u)$. Then the set $\{T_{\tau}u ; \tau\leq0 \}= \{T_{-\tau}u ; \tau\geq0 \}$ is compact in $BC(\R,X)$ since $\{T_{\tau}u ; \tau\geq0 \}$ is also one. Therefore the set $\{T_{\tau}u ; \tau\in\R \}$ is a compact set of $BC(\R,X)$ as the union of two compact sets in $BC(\R,X)$. According to Theorem \ref{th1}, $u$ is periodic.
\vskip 2 mm
{\bf i)} $\Longrightarrow$ {\bf iii)}. 
It is obvious by using Theorem \ref{th1} and the continuity of the restriction operator $R$.
\vskip 2 mm
{\bf iii)} $\Longrightarrow$ {\bf i)}. 
By using Theorem \ref{th1}, we have to prove that $\mathcal{K}:=\{T_{\tau}u ; \tau\in\R \}$ is a compact set of $(BC(\R,X),d_{\infty})$. 
As consequence of Theorem \ref{th2} and Bohner's criterion,
the set $\mathcal{K}$  is relatively compact in $(BC(\R,X),d_{\infty})$ and the function $u$ is almost periodic.
It remains to prove that $\mathcal{K}$ is closed in $(BC(\R,X),d_{\infty})$. 
Let $(\tau_n)_n$ be a sequence of real numbers such that $\ds\lim_{n\to+\infty}d_{\infty}(T_{\tau_n}u,v)=0$. Let us prove that $v=T_{\tau}u$ for some $\tau\in\R$.
By continuity of the operator $R$, we have  $\ds\lim_{n\to+\infty}d_{\infty,+}(R(T_{\tau_n}u),R(v))=0$. 
By assumption, the set $\{R(T_{\tau}u) ; \tau\in\R \}$ is in particular closed in $(BC(\R^+,X),d_{\infty,+})$, then $R(v)=R(T_{\tau}u)$ for some $\tau\in\R$, that is 
\begin{equation}
\label{eq12}
\forall t\geq0 , \qquad v(t)=T_{\tau}u(t)  .
\end{equation}
We have to prove that \eqref{eq12} holds on the whole real line.
The function $T_{\tau}u$ is almost periodic as translation of an almost periodic function and $v$ is also one as uniform limit on $\R$ of almost periodic functions.
Let us denote by $\phi:\R\to\R$ the function defined by $\phi(t):=d(T_{\tau}u(t),v(t))$. 
The function $\phi$ is almost periodic \cite[Property 4, p. 3 \& 7, p.6]{LZ}.
An almost periodic function is uniformly recurrent, then there exists a sequence of real numbers 
such that 
$\ds\lim_{n\to+\infty}\tau_n=+\infty$  and $\ds\lim_{n\to+\infty}\phi(t+\tau_n)=\phi(t)$ for all  $t\in\R$.
From \eqref{eq12}, it follows  $\phi(t)=0$ for all $t\geq0$, so we deduce that  $\ds\phi(t)=\lim_{n\to+\infty}\phi(t+\tau_n)=0$ for all $t\in\R$. Then $v(t)=T_{\tau}u(t)$ for all $t\in\R$.
This ends the proof.
\end{proof}

According Theorem \ref{cor2}, if the set $\{T^+_{\tau}u ; \tau\geq0 \}$ is relatively compact in $BC(\R^+,X)$, then the function u is asymptotically almost periodic. 
We now give an answer to the question what can be said about the function $u$  when $\{T^+_{\tau}u ; \tau\geq0 \}$ is a compact set of $BC(\R^+,X)$.
For $u\in BC(\R^+,X)$, we say that $u$ is \textit{$\omega$-periodic ($\omega>0$) on $[t_0,+\infty)$ for some $t_0\geq0$} if $u(t+\omega)=u(t)$ for all $t\geq t_0$.

\begin{corollary}
\label{cor5}
Let $(X,d)$ be a complete metric space. 
For $u\in BC(\R^+,X)$ the following statements are equivalent.
\vskip 2 mm
{\bf i)} There exists $t_0\geq0$ such that $u$ is $\omega$-periodic on $[t_0,+\infty)$. 
\vskip 2 mm
{\bf ii)} The set $\{T^+_{\tau}u ; \tau\geq0 \}$ is a compact set of  $(BC(\R^+,X),d_{\infty,+})$.
\end{corollary}

\begin{remark} \label{rq1}
Let $u$ be a function which satisfies condition i) of Corollary \ref{cor5}. 
\vskip 2 mm
{\bf i)}
Let us denote by $v\in C(\R,X)$ the  
$\omega$-periodic function satisfying $u(t)=v(t)$ for $t\geq t_0$. A such function $v$ exists and is unique, $v$ is defined by 
$v(t)=u(t-[\frac{t-t_0}{\omega}]\omega)$ where $[\frac{t-t_0}{\omega}]$ denotes the integer part of $\frac{t-t_0}{\omega}$.
\vskip 2 mm
{\bf ii)}
The function $u$  is a special case of asymptotic almost periodic function where the almost periodic function  $v$ is periodic and $d(u(t),v(t))=0$ for $t\geq t_0$.
\end{remark}

\begin{proof}[Proof of Corollary \ref{cor5}] 
{\bf i)} $\Longrightarrow$ {\bf ii)}.   
Let us denote by $v$ the function  defined in Remark \ref{rq1}. 
 By Corollary \ref{cor4} and the periodicity of $v$, we have $\{R(T_{\tau}v) ; \tau\geq  t_0 \}=\{R(T_{\tau}v) ; \tau\in\R \}$
is a compact set of $(BC(\R^+,X),d_{\infty,+})$. 
\vskip 2 mm
First, we have $T^+_{\tau}u=R(T_{\tau}v)$ for $\tau\geq t_0$, then
$\{T^+_{\tau}u ; \tau\geq  t_0 \}$
is a compact set.
\vskip 2 mm
Second, the function $u$ is uniformly continuous on $\R^+$, then the function from $\R^+$ to $BC(\R^+,X)$ defined by $\tau\to T^+_{\tau}u$ is continuous. Then the set 
$\{T^+_{\tau}u ; 0\leq\tau\leq  t_0 \}$ is compact.
\vskip 2 mm
Therefore the set $\{T^+_{\tau}u ; \tau\geq0 \}$ is a compact set of $BC(\R^+,X)$ as the union of two compact set.
\vskip 2 mm
{\bf ii)} $\Longrightarrow$ {\bf i)}. 
As consequence of Theorem \ref{cor2}, the function $u$ is asymptotically almost periodic, that is $\ds\lim_{t\to\infty}d(u(t),v(t))=0$ for some $v\in AP(\R,X)$.  
An almost periodic function is uniformly recurrent, then there exists a sequence of real numbers $(t_n)_n$ such that 
$\ds\lim_{n\to+\infty}t_n=+\infty$ and  $\ds\lim_{n\to+\infty}v(t+t_n)=v(t)$ for all  $t\in\R$.
We deduce that 
\begin{equation}
\label{eq9}
\forall t\in\R , \qquad \lim_{n\to+\infty}u(t+t_n)=v(t) . 
\end{equation}
\vskip 2 mm
{\it First we prove that $v$ is periodic.}
For  $t\in\R$, $\tau_1$, $\tau_2\geq0$, we have for $n$ enough large
$\ds d(u(t+t_n+\tau_1),u(t+t_n+\tau_2)) \leq \sup_{s\geq0}d(u(s+\tau_1),u(s+\tau_2))$.
From \eqref{eq9}, it follows that $\ds \sup_{t\in\R}d(v(t+\tau_1),v(t+\tau_2)) \leq  \sup_{s\geq0}d(u(s+\tau_1),u(s+\tau_2))$ for each $\tau_1$ and $\tau_2\geq0$. According to Lemma \ref{lem2},
 $\set{T_{\tau}v}{\tau\geq0}$ is a  compact set of $(BC(\R,X),d_{\infty})$
 since $\set{T^+_{\tau}u}{\tau\geq0}$ is also one in $(BC(\R^+,X),d_{\infty,+})$.
As consequence of Corollary \ref{cor4}, the function $v$ is periodic.
\vskip 2 mm
{\it Second we prove that: $\exists t_0\geq0$ such that $\forall t\geq0$, $v(t)=u(t+t_0)$.}
By compactness of $\{T^+_{\tau}u ; \tau\geq0 \}$,
there exists a subsequence $(T_{t_{\phi(n})}^+u)_n$  such that $\ds\lim_{n\to+\infty}d_{\infty,+}(T_{t_{\phi(n})}^+u,T_{t_0}^+u)=0$ for some $t_0\geq0$.
From \eqref{eq9} we deduce that $R(v)=T_{t_0}^+u$, that is $v(t)=u(t+t_0)$ for all $t\geq0$. 
\vskip 2 mm
Then $u(t)=v(t-t_0)$ for each $t\geq t_0$ where the function $v(\cdot-t_0)$ is periodic on $\R$.
\end{proof}

Now we give an example of application on dynamical systems of Corollary of \ref{cor4} and Corollary of \ref{cor5}.  For the definition of a dynamical system, see above Corollary \ref{cor3} in Section \ref{ii}.

\begin{corollary} 
\label{cor7}
Let $\left( S(t) \right)_{t\geq0}$ be a dynamical system  on a complete metric space $(X,d)$.
\vskip 2 mm
{\bf i)} If $u$ is a positive trajectory, then $u$ is periodic on $[t_0,+\infty)$ for some $t_0\geq0$ if and only if 
$u(\R^+)$ is a compact set and $\left( S(t) \right)_{t\geq0}$ is  equicontinuous on $u(\R^+)$.
\vskip 2 mm
{\bf ii)} If $u$ is a complete trajectory, then $u$ is periodic if and only if 
$u(\R)$ is a compact set and $\left( S(t) \right)_{t\geq0}$ is  equicontinuous on $u(\R)$.
\vskip 2 mm
{\bf iii)}
There exists a complete trajectory  which is periodic
if and only if there exists a positive trajectory $u$  such that $u(\R^+)$ is a compact set and $\left( S(t) \right)_{t\geq0}$ is  equicontinuous on $u(\R^+)$.
\end{corollary}

\begin{remark}
Thus under the assumption of equicontinuity, a complete trajectory of a dynamical system with a compact range is necessarily periodic, although there are almost periodic functions with a compact range, which are not periodic. An example of  such  function is given by Haraux in \cite{Ha4}.
\end{remark}

\begin{proof}[Proof of Corollary \ref{cor7}]
\vskip 2 mm
{\bf i)} Remark that if  $u$ is a positive trajectory  which  is periodic on $[t_0,+\infty)$ for some $t_0\geq0$, then first $u(\R^+)$ is compact and second
$u\in AAP(\R^+,X)$  (see Remark \ref{rq1}). As consequence of Corollary \ref{cor6}, the set $\left( S(t) \right)_{t\geq0}$ is  equicontinuous on $u(\R^+)$.
Reciprocally assume the positive trajectory $u$ is such that $\left( S(t) \right)_{t\geq0}$ is  equicontinuous on the compact set $u(\R^+)$.
It remains to prove that the positive trajectory 
$u$ is periodic on $[t_0,+\infty)$ for some $t_0\geq0$.
For each $x\in u(\R^+)$, the map $S(\cdot)x$ is continuous and satisfies $S(t)x\in u(\R^+)$ for each $t\geq0$. Then the map $S(\cdot)x$ is bounded and continuous , so the map 
$\Phi:u(\R^+)\to BC(\R^+,X)$ with $\Phi(x)=S(\cdot)x$ is well-defined.
The continuity of $\Phi$ results of the equicontinuity of $\left( S(t) \right)_{t\geq0}$ on $u(\R^+)$. 
 Then the set $\Phi(u(\R^+))=\set{\Phi(u(\tau))}{\tau\geq0}$ is a compact of $BC(\R^+,X)$. 
 Moreover 
  $\Phi(u(\tau))(t)=S(t)u(\tau)=u(t+\tau)$ for $t$ and $\tau\geq0$, then $\Phi(u(\tau))=T^+_{\tau}u$, so
 $\set{T^+_{\tau}u}{\tau\geq0}$ is a compact set  of $BC(\R^+,X)$. According to Corollary \ref{cor5}, the function $u$ is periodic on $[t_0,+\infty)$ for some $t_0\geq0$.
\vskip 2 mm
{\bf ii)} The proof of ii) is similar to that of i) by using Corollary \ref{cor3} instead of \ref{cor6}, Corollary \ref{cor4} instead of Corollary \ref{cor5} and by replacing
the map $\Phi:u(\R^+)\to BC(\R^+,X)$ with $\Phi(x)=S(\cdot)x$ by the map $\Phi:u(\R)\to BC(\R^+,X)$.
This permits to prove that the set $\{\Phi(u(\tau))=R(T_{\tau}u) ; \tau\in\R \}$ is a compact set of $(BC(\R^+,X),d_{\infty,+})$.
\vskip 2 mm
{\bf iii)}
If $v$ is a complete trajectory  which is periodic, then $v(\R)$ is compact and
according to ii), $\left( S(t) \right)_{t\geq0}$ is  equicontinuous on $v(\R)$. So the restriction $u$ of $v$ on $\R^+$  is a positive trajectory  such that $u(\R^+)=v(\R^+)=v(\R)$ since $v$ is periodic, then
$\left( S(t) \right)_{t\geq0}$ is  equicontinuous on the  compact set $u(\R^+)$.
Reciprocally, assume that $u$ is a positive trajectory 
such that $\left( S(t) \right)_{t\geq0}$ is  equicontinuous on the compact set $u(\R^+)$.
According to i), $u$ is $\omega$-periodic on $[t_0,+\infty)$ for some $t_0\geq0$. 
Let us denote by $v$ the function  defined in Remark \ref{rq1}.
For $t\geq s$, there exists  $n_0\in\N$ such that $s+n_0\omega\geq t_0$. 
The function $v$ is $\omega$-periodic and $u$ is a positive trajectory  satisfying $u(\tau)=v(\tau)$ for $\tau\geq t_0$, then
$v(t)=v(t+n_0\omega)=u(t+n_0\omega)=T(t-s)u(s+n_0\omega)=T(t-s)v(s+n_0\omega)=T(t-s)v(s)$ for $t\in\R$ and $n$ enough large.
Then $v$ is a periodic complete trajectory.
\end{proof}

\begin{remark}
Under i) of Corollary \ref{cor7}, one can have $t_0>0$, that is the positive trajectory $u$ is not the restriction of a periodic complete trajectory.
For example, consider the bounded dynamical system  $\left( S(t) \right)_{t\geq0}$ on $L^1(0,1)$ defined by
\begin{equation*}
(S(t)x)(s) = \left\{
\begin{array}
[c]{lll}
x(s-t) &  \text{if} & t<s<1 \\
\text{ \ \ \ \ \ \ \ \ \ \ \ \ \ \ \ \ \ \ \ \ \ }\\
0  & \text{if} & 0<s<t  
\end{array}
\right. 
\end{equation*}
for $x\in L^1(0,1)$ and $0<t<1$. For  $t\geq1$, we set $S(t)=0$.
Then all positive trajectories have a compact range and the alone complete trajectory is the null function.
Thus all positive trajectories are not   the restriction of a periodic complete trajectory except the null function. 
\vskip 2 mm
Not all dynamical systems have this pathology, some systems are such that if two positive trajectories have the same value at the same time, then they are equal. If we consider such systems, we get more refined results from Corollary \ref{cor7}.
\end{remark}

A dynamical system  $\left( S(t) \right)_{t\geq0}$ has the 
\textit{backward uniqueness property}
 if any two positive trajectories having the same value at $t=t_0\geq0$ coincide for any other $t\geq0$. This property is equivalent to $S(t)\in C(X,X)$ is injective for each $t\geq 0$.
We say that \textit{a positive trajectory $u$ is extendable to a periodic complete trajectory}, if there exists a periodic complete trajectory  such that its restriction on $\R^+$ is $u$.

\begin{corollary} 
Let $\left( S(t) \right)_{t\geq0}$ be a dynamical system  on a complete metric space $(X,d)$.
Assume that $\left( S(t) \right)_{t\geq0}$ has the backward uniqueness property.
\vskip 2 mm
{\bf i)} If $u$ is a positive trajectory, then $u$ is periodic on $\R^+$ if and only if  $u(\R^+)$ is a compact set and $\left( S(t) \right)_{t\geq0}$ is  equicontinuous on $u(\R^+)$.
In this case the positive trajectory $u$ is extendable to a periodic complete trajectory $v$.
\vskip 2 mm
{\bf ii)}
If  $v$ is a complete trajectory, then $v$  is periodic
if and only if $v(\R^+)$ is a compact set and $\left( S(t) \right)_{t\geq0}$ is  equicontinuous on $v(\R^+)$.
\end{corollary}

\begin{proof}
{\bf i)} 
The direct implication results of i) of Corollary \ref{cor7}.
For the reciprocal implication we use i) of Corollary \ref{cor7}. Then the positive trajectory $u$ is $\omega$-periodic on $[t_0,+\infty)$ for some $t_0\geq0$.
Let us denote by $v\in C(\R,X)$ the  $\omega$-periodic function satisfying $u(t)=v(t)$ for $t\geq t_0$ (see Remark \ref{rq1}).
The restriction of $v$ on $\R^+$ and $u$ are two  positive trajectories having the same value at $t=t_0$ ($t_0\geq0$).  From  the backward uniqueness property, we have $u(t)=v(t)$ for $t\geq 0$, then $u$ is periodic on $\R^+$. By  build, $v$ is periodic and as in the proof of iii) of Corollary \ref{cor7}, we deduce that $v$ is a complete trajectory.
\vskip 2 mm
{\bf ii)}
The direct implication results of ii)  Corollary \ref{cor7}, since $v(\R^+)=v(\R)$. For the reciprocal implication, we consider $v$  a complete trajectory such that $v(\R^+)$ is a compact set and $\left( S(t) \right)_{t\geq0}$ is  equicontinuous on $v(\R^+)$.
Then the restriction $u$ of the complete trajectory $v$ on $\R^+$ is a positive trajectory  such that 
$u(\R^+)$ is  compact and $\left( S(t) \right)_{t\geq0}$ is  equicontinuous on $u(\R^+)$.
According to i), $u$ is $\omega$-periodic on $\R^+$.
 Let us denote by $w\in C(\R,X)$ the  
$\omega$-periodic
 function  satisfying $u(t)=w(t)$ for $t\geq0$. As in proof of iii)  Corollary \ref{cor7}, we deduce that $w$ is a complete trajectory. Fix $T>0$. The two maps $\tilde{v}$ and $\tilde{w}:\R^+\to X$
defined by  $\tilde{v}=v(\cdot,-T)$ and $\tilde{w}=w(\cdot,-T)$ are two  positive trajectories having the same value at $t=T$.
From  the backward uniqueness property,  we have  $\tilde{v}=\tilde{w}$, that is $v(t)=w(t)$ for $t\geq -T$. Since $T$ is arbitrary, then $v(t)=w(t)$ for each $t\in\R$ where $w$ is a periodic complete trajectory.
This proves that $v$ is a periodic complete trajectory.
\end{proof}


\end{document}